\documentclass[reqno]{amsart}
\usepackage{amssymb,amsthm,amsfonts,amstext,amsmath}
\usepackage{mathtools}
\usepackage{fontawesome5}
\usepackage{eso-pic}
\usepackage{charter}
\usepackage{listings}
\usepackage{color}
\usepackage{textcomp}
\usepackage[normalem]{ulem}

\definecolor{dkgreen}{rgb}{0,0.6,0}
\definecolor{gray}{rgb}{0.5,0.5,0.5}
\definecolor{mauve}{rgb}{0.58,0,0.82}
\lstdefinelanguage{none}{
	identifierstyle=
}

\usepackage{comment}
\usepackage{graphicx}
\usepackage{enumerate}
\numberwithin{equation}{section}
\usepackage{color}
\usepackage{float}
\usepackage[colorlinks=true,linkcolor=blue,citecolor=magenta]{hyperref}

\makeatletter
\def\namedlabel#1#2{\begingroup
	#2%
	\def\@currentlabel{#2}%
	\phantomsection\label{#1}\endgroup
}
\makeatother

\usepackage{mathrsfs}
\usepackage{pythonhighlight}
\usepackage{tikz}
\usepackage{dsfont}
\usepackage{enumitem}

\newlength{\starsize}
\newlength{\starspread}
\tikzset{starsize/.code={\setlength{\starsize}{#1}},
	starspread/.code={\setlength{\starspread}{#1}}}
\tikzset{starsize=1mm,
	starspread=3mm}
\pgfdeclarepatternformonly[\starspread,\starsize]
{custom fivepointed stars}
{\pgfpointorigin}
{\pgfqpoint{\starspread}{\starspread}}
{\pgfqpoint{\starspread}{\starspread}}
{
	\pgftransformshift{\pgfqpoint{\starsize}{\starsize}}
	\pgfpathmoveto{\pgfqpointpolar{18}{\starsize}}
	\pgfpathlineto{\pgfqpointpolar{162}{\starsize}}
	\pgfpathlineto{\pgfqpointpolar{306}{\starsize}}
	\pgfpathlineto{\pgfqpointpolar{90}{\starsize}}
	\pgfpathlineto{\pgfqpointpolar{234}{\starsize}}
	\pgfpathclose%
	\pgfusepath{fill}
}

\usepackage[utf8]{inputenc}
\usepackage[T1]{fontenc}
\usetikzlibrary{backgrounds}
\usetikzlibrary{patterns,fadings}
\usetikzlibrary{arrows,decorations.pathmorphing}
\usetikzlibrary{calc}
\definecolor{light-gray}{gray}{0.95}

\usetikzlibrary{patterns}

\newlength{\hatchspread}
\newlength{\hatchthickness}
\newlength{\hatchshift}
\newcommand{\hatchcolor}{}
\tikzset{hatchspread/.code={\setlength{\hatchspread}{#1}},
	hatchthickness/.code={\setlength{\hatchthickness}{#1}},
	hatchshift/.code={\setlength{\hatchshift}{#1}},
	hatchcolor/.code={\renewcommand{\hatchcolor}{#1}}}
\tikzset{hatchspread=3pt,
	hatchthickness=0.4pt,
	hatchshift=0pt,
	hatchcolor=black}
\pgfdeclarepatternformonly[\hatchspread,\hatchthickness,\hatchshift,\hatchcolor]
{custom north west lines}
{\pgfqpoint{\dimexpr-2\hatchthickness}{\dimexpr-2\hatchthickness}}
{\pgfqpoint{\dimexpr\hatchspread+2\hatchthickness}{\dimexpr\hatchspread+2\hatchthickness}}
{\pgfqpoint{\dimexpr\hatchspread}{\dimexpr\hatchspread}}
{
	\pgfsetlinewidth{\hatchthickness}
	\pgfpathmoveto{\pgfqpoint{0pt}{\dimexpr\hatchspread+\hatchshift}}
	\pgfpathlineto{\pgfqpoint{\dimexpr\hatchspread+0.15pt+\hatchshift}{-0.15pt}}
	\ifdim \hatchshift > 0pt
	\pgfpathmoveto{\pgfqpoint{0pt}{\hatchshift}}
	\pgfpathlineto{\pgfqpoint{\dimexpr0.15pt+\hatchshift}{-0.15pt}}
	\fi
	\pgfsetstrokecolor{\hatchcolor}
	\pgfusepath{stroke}
}

\pgfdeclarepatternformonly[\hatchspread,\hatchthickness,\hatchshift,\hatchcolor]
{custom north east lines}
{\pgfqpoint{\dimexpr-2\hatchthickness}{\dimexpr-2\hatchthickness}}
{\pgfqpoint{\dimexpr\hatchspread+2\hatchthickness}{\dimexpr\hatchspread+2\hatchthickness}}
{\pgfqpoint{\dimexpr\hatchspread}{\dimexpr\hatchspread}}
{
	\pgfsetlinewidth{\hatchthickness}
	\pgfpathmoveto{\pgfqpoint{\dimexpr\hatchshift-0.15pt}{-0.15pt}}
	\pgfpathlineto{\pgfqpoint{\dimexpr\hatchspread+0.15pt}{\dimexpr\hatchspread-\hatchshift+0.15pt}}
	\ifdim \hatchshift > 0pt
	\pgfpathmoveto{\pgfqpoint{-0.15pt}{\dimexpr\hatchspread-\hatchshift-0.15pt}}
	\pgfpathlineto{\pgfqpoint{\dimexpr\hatchshift+0.15pt}{\dimexpr\hatchspread+0.15pt}}
	\fi
	\pgfsetstrokecolor{\hatchcolor}
	\pgfusepath{stroke}
}

\usepackage{float}
\usepackage[colorlinks=true,linkcolor=blue,citecolor=magenta]{hyperref}
\newtheorem{theorem}{Theorem}[section]
\newtheorem{lemma}[theorem]{Lemma}

\newtheorem{proposition}[theorem]{Proposition}

\newtheorem{remark}[theorem]{Remark}

\newcommand{\msf}[1]{{\mathsf #1}}

\newcommand{\bb}[1]{{\mathbb #1}}

\newcommand{\eps}{\varepsilon}

\newcommand{\N}{\mathbb N}
\newcommand{\R}{\mathbb R}

\newcommand{\E}{\mathbb E}

\newcommand{\PP}{\mathbb P}

\def\centerarc[#1](#2)(#3:#4:#5){\draw[#1] ($(#2)+({#5*cos(#3)},{#5*sin(#3)})$) arc (#3:#4:#5);}

\def\II{\mathrm{I\kern-0.1emI}}

\let\oldtocsection=\tocsection
\let\oldtocsubsection=\tocsubsection
\let\oldtocsubsubsection=\tocsubsubsection
\renewcommand{\tocsection}[2]{\hspace{0em}\oldtocsection{#1}{#2}}
\renewcommand{\tocsubsection}[2]{\hspace{1em}\oldtocsubsection{#1}{#2}}
\renewcommand{\tocsubsubsection}[2]{\hspace{2em}\oldtocsubsubsection{#1}{#2}}
\DeclareRobustCommand{\SkipTocEntry}[5]{}


\begin{document}

\title[Delayed Logistic equation as a limit of long memory Markov chains]{Delayed Logistic equation as a limit\\ of  long memory Markov chains}

\author[E. Barros]{Eldon Barros}
\address{IMPA\\
	Estrada Dona Castorina, 110
	Jardim Botânico
	CEP 22460-320
	Rio de Janeiro, Brazil}
\curraddr{}
\email{eldon.barros@impa.br}
\thanks{}

\author[D. Erhard]{Dirk Erhard}
\address{UFBA\\
 Instituto de Matem\'atica, Campus de Ondina, Av. Milton Santos, S/N. CEP 40170-110\\
Salvador, Brazil}
\curraddr{}
\email{dirk.erhard@ufba.br}
\thanks{}

\author[T. Franco]{Tertuliano Franco}
\address{UFBA\\
 Instituto de Matem\'atica, Campus de Ondina, Av. Milton Santos, S/N. CEP 40170-110\\
Salvador, Brazil}
\curraddr{}
\email{tertu@ufba.br}
\thanks{}

\author[M. Jara]{Milton Jara}
\address{IMPA\\
	Estrada Dona Castorina, 110
	Jardim Botânico
	CEP 22460-320
	Rio de Janeiro, Brazil}
\curraddr{}
\email{mjara@impa.br}
\thanks{}


\begin{abstract}	
	We introduce and analyze a long memory continuous-time Markov chain on $\bb R_+$ whose jump mechanism depends explicitly on a state in the past. From the present state $x_0$, the process jumps to $x_0(1+\frac{1}{N})$ or to $x_0(1-\frac{x_{-\lfloor \tau N\rfloor}}{N^2})$, each at rate $1/2$, where $x_{-\lfloor \tau N\rfloor}$ denotes the state located $\lfloor \tau N\rfloor$  jumps backward in time. Here the delay $\tau>0$ is fixed and $N$ is the scaling parameter. The initial condition is prescribed by a vector of length $\lfloor \tau N\rfloor+1$, all of whose entries are equal to $uN$.
		Using a genuine space–time replacement lemma, we prove that, as $N\to\infty$, the rescaled process converges to a deterministic limit governed by the \textit{Delayed Logistic Equation} (also known as the \textit{Hutchinson equation}) with delay $\tau$ and initial condition $\rho(t)\equiv \mu$ for $t\in[-\tau,0]$. 
\end{abstract}

%

\maketitle

\tableofcontents

\allowdisplaybreaks

\section{Introduction}\label{s1}

Delayed differential equations are functional differential equations in which the evolution of the system depends not only on its present state, but also on its past. For general introductions we refer to \cite{Hale} and \cite{Ruan}. Delay equations play a central role in modeling phenomena where reaction times, maturation periods, or information propagation are non-negligible. Prominent applications arise in infectious disease dynamics, population biology, neuronal networks, economics, and finance; see \cite{Hal, Rihan2021} for examples.

A particularly important example is the \textit{Delayed Logistic Equation}, also known as \textit{Hutchinson Equation},  which generalizes the classical \textit{Logistic Equation}. The latter was introduced by Verhulst in 1838  (see \cite{verhulst1838notice}) and is given by
\[
\begin{cases}
	u'(t)\; =\; r u(t)\Big(1 - \frac{u(t)}{K}\Big)\,, & \text{ for } t>0\,,\\
	 u(0) = \rho\,, &
\end{cases}
\]
as a model for population growth. The growth rate is proportional to the population size and modulated by the factor  $1 - u(t)/K$, which reflects the competition for finite food resources.
In 1948, Hutchinson \cite{HUTCHINSON}
proposed the following delayed version of the logistic equation
\begin{equation}\label{eq:Delayed}
\begin{cases}
u'(t)\; =\; r u(t)\Big(1 - \frac{u(t-\tau)}{K}\Big)\,, & \text{ for } t>0\,,\\
u(t) = \rho(t)\,, & \text{ for } t \in [-\tau, 0]\,,
\end{cases}
\end{equation}
where the delay $\tau>0$ models that the effect of reduced resources is not instantaneous but is felt only after a time lag $\tau$. In contrast to the classical logistic equation, the initial condition is now a function on $[-\tau,0]$. The delayed logistic equation exhibits a rich phenomenology, including oscillatory behavior absent in the non-delayed case, and a Hopf bifurcation.

Despite its classical status and extensive analytical study, a fundamental probabilistic question remains largely open:
{\em Can delayed differential equations arise as the limit of microscopic stochastic dynamics?}
In this work we answer this question affirmatively for the Delayed Logistic Equation \eqref{eq:Delayed}.


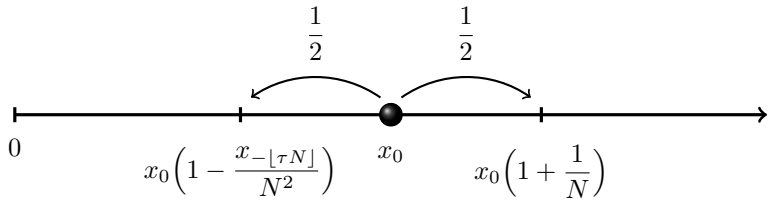
\begin{figure}[!htb]
	\centering
	\begin{tikzpicture}[scale=1]
		\draw[very thick, ->] (0,0)--(10,0);
		\draw[very thick] (0,0.1) -- (0,-0.1); 
		\draw (0,-0.2) node[below]{$0$};

		\begin{scope}[xshift = 1cm]
			\draw (3,0.6) node[above]{$\displaystyle\frac{1}{2}$};
			\draw (5,0.6) node[above]{$\displaystyle\frac{1}{2}$};		
			\centerarc[thick,->](3,-1)(55:125:1.5);
			\centerarc[thick,<-](5,-1)(55:125:1.5);

			\draw (4,-0.3) node[below]{$x_0$};
			\draw[very thick] (6,0.1) -- (6, -0.1);
			\draw (6,-0.75)node[]{$x_0\Big(1+\dfrac{1}{N}\Big)$};
			\draw[very thick] (2,0.1) -- (2,-0.1);
			\draw (2,-0.75)node[]{$x_0\Big(1-\dfrac{x_{-\lfloor \tau N\rfloor}}{N^2}\Big)$};
			\filldraw[ball color=black] (4,0) circle (.15);
		\end{scope}

	\end{tikzpicture}
	
	\bigskip
	\caption{Jump rates for the \textit{long memory Markov chain} considered here. Note that the origin is an absorbent state. Above, the real value $x_0$ denotes the present state while $x_{-\lfloor \tau N\rfloor}$ denotes the state $\lfloor \tau N\rfloor$ jumps backward in time.}
	\label{fig:rates}
\end{figure}

In this work we introduce and analyze a continuous-time stochastic process with long memory on the positive half-line whose jump mechanism is illustrated in Figure~\ref{fig:rates}. Although the process is not Markovian in its natural state variable, it admits a Markovian representation in an enlarged state-space. The transition mechanism depends explicitly on a state $\lfloor \tau N\rfloor$ jumps in the past, thereby encoding a memory window of macroscopic length $\tau$ under the appropriate scaling.

Our main result shows that, under proper rescaling, this long memory stochastic dynamics converges in probability to the solution of the Delayed Logistic Equation~\eqref{eq:Delayed}. Thus, macroscopic delay emerges purely from microscopic memory effects. 
Convergence of trajectories of Markov chains on $\bb R$ towards solutions of ordinary differential equations -- commonly referred to as a \textit{fluid limit} -- is by now well understood; see \cite{EK}. In contrast, the derivation of a delay differential equation as a scaling limit of a microscopic stochastic dynamics appears to be largely unexplored. In particular, to the best of our knowledge, this is the first instance in which a delayed differential equation arises as the  limit of a stochastic process with memory. We emphasize that the scaling considered here is related to the fluid limit, but does not coincide with it. In fluid limit  time and space are rescaled, while in our setting  time, space \textit{and memory} (or state-space, as we shall see) are rescaled. 

The emergence of delay at the macroscopic level is intimately connected with the presence of long memory at the microscopic scale. While classical Markov chains evolve without reference to their past, the present dynamics here depend explicitly on a state located $\lfloor \tau N\rfloor$ jumps backward in time. Under our scaling, this discrete memory window becomes a fixed delay $\tau$ in the limiting equation. 

From a technical perspective, the proof requires some extensions of standard fluid limit techniques. The first step consists in embedding the long memory process into a Markov chain on a high-dimensional state-space. We then analyze Dynkin’s martingale associated with a suitable observable. As in the classical approach, the strategy is to prove that this martingale vanishes uniformly in probability, leading to an integral formulation of the limiting equation.

Two major difficulties arise. The first concerns the quadratic variation of the martingale. To overcome this, we develop moment estimates tailored to the enlarged state-space. The second difficulty is to obtain the correct limiting equation, which requires a replacement argument inspired by Varadhan's entropy method, see \cite{kl}.
In contrast to the standard replacement lemma which deals with time integrals of space averages, we must establish a genuine \textit{space-time replacement} lemma. More precisely, we must show that, in a suitable sense,  $x_0(t-\tau)\approx x_{-\lfloor \tau N\rfloor}(t)$. That is, we must assure that \textit{the past in time  is approximately equal to the past in space}.  The precise formulation is given in Section~\ref{s3}.

Our hope is that the present paper serves as a starting point to study the emergence of delay from a microscopic dynamic. It would be further interesting to investigate if similar results for fluctuations could be obtained.
 
The paper is divided as follows. In Section~\ref{s2} we define the model and state the main results. Section~\ref{s3} contains the proofs. Finally, in Section~\ref{s4} we present computational simulations of the model, illustrating the convergence towards the delayed logistic dynamics.

\section{Model and Statement of Results}\label{s2}
Denote by $\bb R_{\geq 0} = [0,\infty)$ the set of nonnegative reals numbers and assume that $\bb N = \{1,2,\ldots\}$.
Given two real valued functions $f,g$ defined on the same space $X$, we will write hereinafter $f(u) \lesssim g(u)$ if there exists a constant $C$ independent of $u$ such that $f(u) \le C g(u)$ for every $u\in X$.

 Fix parameters $\tau, \mu \in \R_{\geq 0}$ and $N \in \N$. We next construct a continuous-time Markov chain $X^N$ that tracks both the current population and its recent history.

For ease of notation,  we will often denote  $\lfloor \tau N\rfloor$ simply by $\tau N$.
Consider the  state-space $\Omega_N = \bb R_{\geq 0}^{ \tau N +1}$. Elements of this state-space will be denoted by $x = (x_{-\tau N}, \ldots, x_0)$, where $x_0$ represents the current population size and $x_{j}$ the population size $-j$ time steps in the past.   Define the shift operators $\theta_N^+$ and $\theta_N^-$ as the operators  on $\Omega_N$ such that, for  $x \in \Omega_N$, 
\begin{equation}\label{eq:rates+}
	\big(\theta_N^+(x)\big)_{j} \;:=\;
	\begin{cases}
		x_{j+1}, & \text{if }  -\tau N \leq j<0, \vspace{3pt}\\
		x_0\Big(1+\frac{1}{N}\Big), & \text{if } j = 0, 
	\end{cases}
\end{equation}
and
\begin{equation}\label{eq:rates-}
	\big(\theta_N^-(x)\big)_{j} \;:=\;
	\begin{cases}
		x_{j+1}, & \text{if }  -\tau N \leq j<0, \vspace{3pt}\\
		x_0\Big(1-\frac{x_{-\tau N }}{N^2}\Big) \vee 0, & \text{if } j = 0. 
	\end{cases}
\end{equation}

The vectors $\theta_N^{+}(x)$ and $\theta_N^{-}(x)$ represent the state $x$ after a birth or death event, respectively, together with the shift of the memory window.
Let $\xi^{N} = (\xi^{N}(n))_{n \in \N}$ be the discrete-time Markov chain on $\Omega_{N}$ with transition probabilities 
\begin{equation}\label{eq:discrete_Markov}
	P_{N}(x,y) \;=\;\frac{1}{2}\delta_{\theta_N^{-}(x)}(y) +\frac{1}{2} \delta_{\theta_N^{+}(x)}(y) \,,
\end{equation}
and initial condition given by $\xi_{j}^{N}(0) =  \mu N$ for all $j \in \{ - \tau N, \ldots ,0\}$.

Let $\sigma = \left(\sigma_{n}\right)_{n\geq1}$ be an i.i.d sequence of exponential random variables with parameter $1$ and define the jumps times $J_{n} = \sum_{j=1}^{n} \sigma_{j}$. The continuous-time process $X^{N}= (X^{N}(t))_{t\geq 0} = (X^{N}_{-\tau N}(t), \ldots, X^{N}_{0}(t))_{t\geq 0}$ taking values in $\Omega_N$ is defined by 
\begin{equation}\label{chain}
	X^{N}(t) \;=\; \sum_{n \in \bb N}\xi^{N}(n) \mathds{1}_{\left\{t \in \left[J_{n},J_{n+1}\right)\right\}}, \quad \text{for } t>0, 
\end{equation}
with $X^{N}(0) = \xi^{N}(0)$. The infinitesimal  generator $\msf{L}_N$ of $X^{N}$ acts on bounded measurable functions $f: \Omega_N \rightarrow \R$ as
\begin{equation*}
	\msf{L}_Nf(x) \;=\; \frac{1}{2}\left(f(\theta_N^{+}(x))- f(x)\right)+ \frac{1}{2}\left(f(\theta_N^{-}(x))- f(x)\right).	
\end{equation*}

\begin{theorem}\label{thm:main}
	Fix $\tau > 0$ and $\mu \in (0,1)$. For each $N \in \N$, let $X^N$ be the continuous-time Markov chain defined in \eqref{chain}. Define the scaled process
	\begin{equation}\label{eq:Y}
		Y^N(t) \;=\;
		\begin{cases}
			\dfrac{X_{- \tau N}^N(N(t+\tau))}{N}\,, & \text{if } t \in [-\tau, 0), \vspace{5pt}\\
			\dfrac{X_0^N(Nt)}{N}\,, & \text{if } t \geq 0.
		\end{cases}
	\end{equation}
	Then, for every $T > 0$, the sequence of processes $\{Y^N(t):t\in[-\tau,T]\}$ converge in probability in the Skorohod topology of $D([-\tau,T], \bb R)$ to the unique continuous solution $u:[-\tau,T]\to \bb R$ to the Delayed Logistic  differential equation
	\begin{equation}\label{eq:dde}
		\begin{cases}
			\displaystyle u'(t) \;=\;\frac{1}{2} u(t)\Big(1 - u(t-\tau)\Big)\,, & \text{ for } t > 0,\vspace{4pt}\\
			u(t) \;=\; \mu\,, &  \text{ for } t \in [-\tau, 0].
		\end{cases}
	\end{equation}
\end{theorem}
\begin{remark}
\rm Before turning to the proofs we introduce an additional notion.  Recall that the state-space of the Markov chain $X^N$ is $\Omega_N = \bb R_{\geq 0}^{ \tau N +1}$ and each of the two possible transitions involves a shift of a configuration  $x = (x_{-\tau N}, \ldots, x_0)$ within that space. We adopt the following terminology. By \textit{past in time} we mean, as usual, an earlier time along the trajectory of the Markov chain. By \textit{past in space}, we refer to one of the coordinates of the vector $X^N(t) = \big(X^N_{-\tau N}(t), \ldots, X^N_0(t)\big)$, where $X^N_0(t)$ represents the present. Note that, due to the fact that any transition comes with a shift in space, any \textit{past in space} has been the present at some \textit{past in time}.
\end{remark}
\section{Proofs}\label{s3} 

\subsection{Proof of the \texorpdfstring{Theorem~\ref{thm:main}}{Theorem 2.1}}
In this section we prove the Theorem~\ref{thm:main} subjected to several results to be proven in the forthcoming sections. Consider the function $f_N:\Omega_N\to\R$ given by $f_{N}(x) = x_{0}/N$. By the  Dynkin's martingale formula, we know that the process
\begin{equation}\label{dynkinmartingale}
	M_t^N \;:=\; f_N(X^N(t)) - f_N(X^N(0)) - \int_0^t \msf{L}_N f_N(X^N(s))\,ds
\end{equation}
is a càdlàg martingale.
Define the process
\begin{equation}\label{eq:defZ}
	Z^N(t) \;:=\; \frac{X^N_{-\tau N}(Nt)}{N}\,,
\end{equation}
which represents the \textit{past in space}. Using \eqref{dynkinmartingale}  and \eqref{eq:defZ}, we obtain the following semimartingale decomposition for the rescaled process $Y^N$ defined in \eqref{eq:Y}:
\begin{equation}\label{eq:semimartingaledecomposition}
	\begin{aligned}
		Y^N(t)
		&\;=\; Y^N(0)
		+ \frac{1}{2} \int_0^t Y^N(s)\bigl[1 - Y^N(s-\tau)\bigr]\,ds \\
		&\quad + \frac{1}{2} \int_\tau^t Y^N(s) D^N(s)\,ds
		+ \frac{1}{2} \int_0^t Y^N(s) \widetilde{Z}^N(s)\,ds
		+ M_{Nt}^N\,,
	\end{aligned}
\end{equation}
where
\begin{equation}\label{eq:diff_Y_Z}
	D^N(s) := Y^N(s-\tau) - Z^N(s)\,,
\end{equation}
and
\begin{equation}\label{Ztilde}
	\widetilde{Z}^N(s) := Z^N(s) - \bigl(Z^N(s) \wedge N\bigr)\,.
\end{equation}
By Proposition~\ref{lem:tight}, the sequence of processes $\{Y^N(t) : t \in [-\tau,T]\}_{N\geq 1}$ is tight in the space $D([-\tau,T], \bb R)$. Hence, along a subsequence still denoted by $N$, the sequence of processes $Y^N$ converges in distribution in the Skorokhod topology to a process which we denote by $\{Y(t) : t \in [-\tau,T]\}$. 

We now identify the limit. By Proposition~\ref{lem:initialcondition}, we have that $Y(t) = \mu$ for $t \in [-\tau,0]$. Now, Proposition~\ref{lem:jumps} together with \cite[Theorem 13.4]{Bill} implies that $Y$ is almost surely continuous and we obtain that $Y^N(s)\bigl[1 - Y^N(s-\tau)\bigl]$ converges uniformly to $Y(s)\bigl[1 - Y(s-\tau)\bigr]$. Then, the continuous mapping theorem yields
\begin{equation*}
	\lim_{N\to\infty} \frac{1}{2}\int_0^t Y^N(s)\bigl[1 - Y^N(s-\tau)\bigr]\,ds
	\;=\;
	\frac{1}{2}\int_0^t Y(s)\bigl[1 - Y(s-\tau)\bigr]\,ds
\end{equation*}
in distribution in the space $D([0,T], \bb R)$.
Furthermore, by Proposition~\ref{lem:replacement},
\begin{equation*}
	\lim_{N\to\infty} \int_\tau^t Y^N(s) D^N(s)\,ds \;\xrightarrow[N\to\infty]{\mathbb{P}}\; 0\,,
\end{equation*}
and by Proposition~\ref{lem:Ztildelem},
\begin{equation*}
	\int_0^t Y^N(s)\widetilde{Z}^N(s)\,ds \;\xrightarrow[N\to\infty]{\mathbb{P}}\; 0\,.
\end{equation*}
Finally, Proposition~\ref{lem:martingale} ensures that the martingale term $M_{Nt}^N$ vanishes in probability. Passing to the limit in \eqref{eq:semimartingaledecomposition}, we conclude that, for all $t \in [0,T]$,
\[
Y(t) \;=\; Y(0) + \frac{1}{2}\int_0^t Y(s)\bigl[1 - Y(s-\tau)\bigr]\,ds\,.
\]
That is, $Y$ is a solution of the Hutchinson equation with initial condition $Y(t) = \mu$ for $t\in [-\tau,0]$, which is well-posed (see the discussion in Section 5.5 of~\cite{Hal}). Hence, the sequence of processes $\{Y^N(t):\, t\in [-\tau, T]\}_{N\geq 1}$  converges in distribution in the Skorokhod topology $D([-\tau,T], \bb R)$ to the unique solution of~\eqref{eq:dde}. Finally, since the limiting object is deterministic, convergence in distribution implies convergence in probability,  completing the proof of the Theorem~\ref{thm:main}.


\subsection{Uniform moment bounds}
We start with some moment estimates on the processes $Y^N(t)$, $Z^N(t)$ and $\widetilde{Z}^N(t)$. These bounds will be useful in the sequel.

Let $F(t)$ be the counting process which gives the number of jumps of the\break continuous-time Markov chain $\{X^N(t): t\in [0,T]\}$ up to time $t\geq 0$. Since  the waiting times between two jumps are i.i.d.\ exponential random variables with parameter $1$, the process $F$ is a $\mathtt{PPP}(1)$. Thus, for fixed time $t>0$, the random variable $F(t)$ has Poisson distribution with parameter $t$.
\begin{lemma}\label{lem:momentsbound1}
	For each $p \geq 1$ and $N \in \N$, there exists a constant $K_{N} \in (1, 1+1/N)$ such that 
	\begin{equation*}\label{eq:moment-bound2}
		\E\left[\bigg(\sup_{t \in [0,T]} Y^{N}(t)\bigg)^p\right] \;\leq\; \mu^p \exp\left\{pK_{N}^{p-1}T\right\}.
	\end{equation*}
In particular,
\begin{equation*}
	\sup_N \sup_{t \in [0,T]}\mathfrak{m}_p^N(t) \;<\; \infty\,,
\end{equation*}
where $\mathfrak{m}_p^N(t):=\E[(Y^N(t))^p]$.
\end{lemma}

\begin{proof}
Bounding by the extreme  case where we have only  jumps to the right,  it holds almost surely that
	\begin{equation*}
		\bigg(\sup_{t \in [0,T]} Y^{N}(t)\bigg)^p  \;\leq\; \sup_{t \in [0,T]} \mu^{p}\left(1 + 1/N\right)^{pF(Nt)} \;=\; \mu^{p}\left(1 + 1/N\right)^{pF(NT)}.
	\end{equation*}
	Now using exponential moments of the Poisson distribution and the Mean Value Theorem, we obtain
	\begin{equation*}
		\begin{split}
			\E\left[\bigg(\sup_{t \in [0,T]} Y^{N}(t)\bigg)^p\right] & \;\leq\; \mu^{p} \E\left[\left(1 + 1/N\right)^{pF(NT)} \right] =  \mu^{p} \E\left[e^{F(NT)p\log(1+1/N)} \right] \\ & \;=\; \mu^{p}\exp\big\{NT[(1+1/N)^p-1]\big\} = \mu^{p}\exp\big\{pK_{N}^{p-1}T\big\}
		\end{split}
	\end{equation*}
	for some $K_{N} \in \left(1, 1 + 1/N\right)$, finishing the proof.
\end{proof}
Recall the definition \eqref{eq:defZ} of $Z^N(t)$ .
\begin{lemma}\label{lem:momentsbound2}
	For each $p \geq 1$ and $N \in \N$, there exists a constant $K_{N} \in (1, 1+1/N)$ such that 
	\begin{equation*}
		\E\left[\bigg(\sup_{t \in [0,T]} Z^{N}(t)\bigg)^p\right] \;\leq\; (2\mu)^p \exp\left\{pK_{N}^{p-1}T\right\}.
	\end{equation*}
	In particular,
\begin{equation*}
	\sup_N \sup_{t \in [0,T]}\widetilde{\mathfrak{m}}_p^N(t) \;<\; \infty\,,
\end{equation*}
where $\widetilde{\mathfrak{m}}_p^N(t):=\E[(Z^N(t))^p]$.
\end{lemma}

\begin{proof}
	Recall that by~\eqref{chain} we have that $X^N(t)= \xi^N(n)$ if and only if the process performed exactly $n$ jumps until time $t$. Recall also that the transitions of the process involve a shift in space and a jump to the left or a jump to the right  at the zero coordinate  of the process. Utilizing the fact that in the extreme case we have only jumps to the right, it holds almost surely that 
	\begin{equation*}
		\begin{split}
			&\bigg(\sup_{t \in [0,T]} Z^{N}(t)\bigg)^p\\
			 &=\; \bigg( \sup_{t \in [0,T]} N^{-1}\left[\xi_{0}^{N}(F(Nt)-\tau N) \mathds{1}_{\{ F(Nt)> \tau N \}} + N \mu \mathds{1}_{\{ F(Nt) \leq \tau N \}} \right] \bigg)^p \\ 
			 &\leq\; \mu^{p} \left[ \left(1+ 1/N\right)^{F(NT)-\tau N} + 1\right]^p \leq (2\mu)^{p} (1+ 1/N)^{pF(NT)}\,.
		\end{split}
	\end{equation*}
	Then, using again exponential moments of the Poisson distribution and the Mean Value Theorem, we obtain that, for some $K_{N} \in \left(1, 1 + 1/N\right)$,
	\begin{equation*}
			\E\left[\bigg(\sup_{t \in [0,T]} Z^{N}(t)\bigg)^p\right] \;\leq\; (2\mu)^p \exp\left\{pK_{N}^{p-1}T\right\}.
	\end{equation*}
Hence, we can conclude the proof.
\end{proof}
Recall the definition \eqref{Ztilde} of $\widetilde{Z}^{N}(s)$.  
\begin{lemma}\label{lem:momentsbound3}
	For each $p\geq 1$ and $q \geq 1$, there exists a constant $C_{p,q,T}>0$ depending only on $p,q$ and $T$ such that 
	\begin{equation*}
		\E\left[\left( \widetilde{Z}^{N}(s)\right)^{p} \right] \;\leq\; C_{p,q,T} N^{-q/2}
	\end{equation*}
	for all $s \in [0,T]$.
\end{lemma}
\begin{proof}
	Using the Cauchy-Schwarz inequality, we obtain the bound
		\begin{equation*}
			\begin{split}
				\E\left[\left( \widetilde{Z}^{N}(s)\right)^{p} \right]  &= \E\left[ \bigg(Z^{N}(s) - N\bigg)^{p}\mathds{1}_{\left\{X_{-\tau N}^{N}(Ns) > N^{2}\right\}}\right] \\ 
				&\leq \E\left[ Z^{N}(s)^{p}\mathds{1}_{\left\{X_{-\tau N}^{N}(Ns) > N^{2}\right\}}\right] \\
				&\leq \E\big[Z^N(s)^{2p}\big]^{1/2}\Big(\PP\left(X_{-\tau N}^{N}(Ns) > N^{2} \right) \Big)^{1/2}\,.
			\end{split}
		\end{equation*}
		Now, to estimate the probability above we use Markov's inequality, which yields the bound
		\begin{equation*}
			\PP\left(X_{-\tau N}^{N}(Ns) > N^{2} \right) \;\leq\; \frac{\E[X_{-\tau N}^{N}(Ns)^q]}{N^{2q}}\;=\; \frac{1}{N^q} \E[Z^{N}(s)^q]\,.
		\end{equation*}
		Recalling Lemma~\ref{lem:momentsbound2}, we can deduce that there exists a constant $C_{p,q}$ such that 
		\begin{eqnarray*}
			\E\left[\left( \widetilde{Z}^{N}(s)\right)^{p} \right]\;\leq\; C_{p,q}N^{-q/2}\,, \qquad\forall\, N \in \bb N,
		\end{eqnarray*}
		concluding the proof.
\end{proof}


In possess of the previous moment bounds, we can study the convergence of the Dynkin martingale \eqref{dynkinmartingale}.
\begin{proposition} \label{lem:martingale}
Consider the  martingale $M^N_t$ given by~\eqref{dynkinmartingale}. Then,
\begin{equation*}
	\lim_{N\to\infty} \sup_{t\in [0,T]}|M_{Nt}^{N}| \;=\; 0
\end{equation*}
in probability.
\end{proposition}
\begin{proof}
The predictable quadratic variation is given by
\begin{equation*}
	\begin{split}
		\langle M^{N} \rangle_{Nt} & \;=\; \frac{1}{2N} \int_{0}^{t}\left[\left(Y^{N}(s)\right)^2 + \left(Y^{N}(s)Z^{N}(s)\right)^2\right]ds\,.
	\end{split}
\end{equation*}
Taking expectation on both sides, applying Fubini's Theorem and using Cauchy-Schwarz inequality, we obtain
\begin{equation*}
	\begin{split}
		\mathbb{E}\left[\langle M^{N} \rangle_{Nt} \right] &\;\leq\; \frac{1}{2N} \int_{0}^{t}\left[\mathfrak{m}_{2}^{N}(s) + (\mathfrak{m}_{4}^{N}(s)\widetilde{\mathfrak{m}}_{4}^{N}(s))^{1/2}\right]ds\,.
	\end{split}
\end{equation*}
Thus, by Lemmas  \ref{lem:momentsbound1} and \ref{lem:momentsbound2} we may conclude that $\mathbb{E}\left[\langle M^{N} \rangle_{Nt} \right] \rightarrow 0$ as $N \rightarrow +\infty$. Therefore, using the Doob's Inequality we conclude that $$\lim_{N\to\infty} \sup_{t\in [0,T]}|M_{Nt}^{N}| \;=\; 0$$ in probability.
\end{proof}
\subsection{Replacement lemma}
We start with a bound on the increments of the rightmost entry $\xi^N_0(t)$ of the discrete-time Markov chain $\xi^N(t)$ defined at \eqref{eq:discrete_Markov}.
\begin{lemma}\label{lem:timesdifferences}
	For each $m,q \in \N_{0}$ and $T>0$, it almost surely holds that
	\begin{equation*}
		|\xi_{0}^{N}(m+q)-\xi_{0}^{N}(m)| \;\leq\; q H(T)\,,
	\end{equation*}
		where
	\begin{equation}\label{eq:HN}
		H^{N}(T) \;:=\; \sup_{{s} \in [0,T]}\left\{ Y^{N}({s})+ Y^{N}({s})Z^{N}({s})\right\}.
	\end{equation}
\end{lemma}
\begin{proof}
	Each jump multiplies the current population by either $1+1/N$ or by $1 - x_{- \tau N }/N^2$. In particular, the jump factor is  at most $1 + 1/N$ and at least $1 - x_{- \tau N }/N^2$. Therefore, for each $n\in \N_{0}$,
	\begin{equation*}
			|\xi_{0}^{N}(n+1) - \xi_{0}^{N}(n)| \;\leq\; \xi_{0}^{N}(n)\max\left\{\frac{1}{N}, \frac{\xi_{- \tau N }^{N}(n)}{N^2}\right\}\;\leq\; \xi_{0}^{N}(n)\left(\frac{1}{N} + \frac{\xi_{- \tau N }^{N}(n)}{N^2} \right).
	\end{equation*}
	Thus, for any $m,q \in \N_{0}$,
	\begin{equation*}
		\begin{split}
			|\xi_{0}^{N}(m+q) - \xi_{0}^{N}(m)| & \;\leq\; \sum_{j=0}^{q-1}\left|  \xi_{0}^{N}(m+1+j) - \xi_{0}^{N}(m+j) \right| \\ &\;\leq\; \sum_{j=0}^{q-1} \xi_{0}^{N}(m+j)\left( \frac{1}{N} +  \frac{\xi_{-\tau N}^{N}(m+j)}{N^2}\right). 
		\end{split}
	\end{equation*}
	Since each parcel in the sum above is bounded from above by the supremum of the chain up to time $NT$, we obtain that
	\begin{equation*}
		|\xi_{0}^{N}(m+q)-\xi_{0}^{N}(m)| \;\leq\; q H(T)\,,
	\end{equation*}
	concluding the proof.
\end{proof}

	Recall from \eqref{eq:diff_Y_Z} that 
		\[D^N(s)\; :=\; Y^{N}(s-\tau)-Z^{N}(s)  \;=\;  \frac{1}{N}\Big(X_{0}^N(N(s-\tau))-X_{- \tau N}^N(Ns)\Big)\,.\]

\begin{proposition}[Replacement Lemma]\label{lem:replacement}
	Define the integrated error
	\begin{equation*}
		I_N(t) \;:=\; \frac{1}{2}\int_{\tau}^t Y^N(s) D^N(s)\,ds\,.
	\end{equation*}
	Then, for every $T > \tau$,
	\begin{equation*}
	\lim_{N\to\infty} \sup_{t\in [\tau,T]} |I_N(t)| = 0
	\end{equation*}
	in probability.
\end{proposition}

\begin{proof}

 By the definitions of $X^{N}$ and $F(t)$, we have
	\begin{equation*}
		X_{j}^{N}(t) \;=\; \xi_{j}^{N}(F(t)) \;=\; \xi_{0}^{N}(F(t)+j)\mathds{1}_{\{ F(t) + j > 0  \}} + N \mu \mathds{1}_{\{ F(t) + j \leq 0 \}} 
	\end{equation*}
	for all integers $j \leq 0$. 	Then, we may write	
	\begin{equation*}
		\begin{split}
			|D^{N}(s)| &\;=\; \frac{1}{N} \left|\xi_{0}^{N}(F(N(s-\tau))) - \xi_{- \tau N }^{N}(F(Ns))  \right| \\
			& \;\leq\; \frac{1}{N} \left|\xi_{0}^{N}(F(N(s-\tau))) - \xi_{0}^{N}(F(Ns)-\lfloor \tau N \rfloor)  \right| \mathds{1}_{\left\{ F(Ns) > \lfloor \tau N \rfloor\right\}} \\ &\quad +  \frac{1}{N} \left|\xi_{0}^{N}(F(N(s-\tau))) - \mu N \right| \mathds{1}_{\left\{ F(Ns) \leq \lfloor \tau N \rfloor\right\}}.
		\end{split}
	\end{equation*}
	Applying Lemma~\ref{lem:timesdifferences}, first with $q= q_{N}^{s}(\tau) = |F(Ns) - F(N(s-\tau)) - \lfloor \tau N \rfloor |$ and $m = m_{N}^{s}(\tau) = F(N(s-\tau)) \wedge (F(Ns)-\lfloor \tau N \rfloor)$ for the first term then with $q = F(N(s-\tau))$ and $m = 0$ for the second term, it almost surely holds that
	\begin{equation*}
		|D_{N}(s)| \;\leq\; H^{N}(T)\Big[N^{-1}q_{N}^{s}(\tau) + N^{-1}F(Ns)\mathds{1}_{\{ F(Ns) \leq \tau N\}}\Big].
	\end{equation*}
	Thus, for each $p\geq1$,
	\begin{equation*}
			|D_{N}(s)|^{p} \;\lesssim\; (H^{N}(T))^{p}\Big[N^{-p}(q_{N}^{s}(\tau))^{p} + N^{-p}(F(Ns))^{p} \mathds{1}_{\{ F(Ns) \leq \tau N\}}\Big].
	\end{equation*}
	Taking expectation on both sides, applying Cauchy-Schwarz inequality, using Lemmas  \ref{lem:momentsbound1} and \ref{lem:momentsbound2}, and a moment estimate for the Poisson distribution, we obtain
	\begin{equation}\label{Dmoments}
			\E\big[|D_{N}(s)|^{p}\big] \;\lesssim\; N^{-p}\Big(\E\left[(q_{N}^{s}(\tau))^{2p}\right]\Big)^{1/2} + \PP( F(Ns) \leq \tau N)\,.
	\end{equation}
	Now, we need to estimate both terms in the inequality \eqref{Dmoments}, we start with the probability term. By an application of Markov's inequality, for each $\theta > 0$,
	\begin{equation}\label{eq:Poissontail}
		\begin{split}
			\mathbb{P}\left( F(Ns) \leq  \tau N  \right) &\;=\; \mathbb{P}\left( e^{-\theta F(Ns)} \geq e^{-N\tau\theta}\right) \\ 
			&\;\leq\; e^{N \tau \theta}\E\left[e^{-\theta F(Ns)}\right] \;=\; \exp\left(N\tau \theta + Ns(e^{-\theta}-1)\right).
		\end{split}
	\end{equation}
	Optimizing over $\theta>0$, we obtain
	\begin{equation*}
		\mathbb{P}\left( F(Ns) \leq  \tau N  \right)  \;\leq\;  e^{-Nc(s,\tau)}\,,
	\end{equation*}
	where $c(s,\tau) = \tau(s/\tau-1- \log(s/\tau))$.
	Since  $x-1-\log(x)>0$ for $x>1$, the probability above decays exponentially provided $s>\tau$. To estimate the first term in the last inequality of \eqref{Dmoments}, define $q_{N}(\tau) = |F(N\tau) - N\tau|$ and $k_{N}^{s}(\tau) = q_{N}^{s}(\tau) - q_{N}(\tau)$. Then,
	\begin{equation*}
			\E\left[(q_{N}^{s}(\tau))^{2p}\right] \;\leq\; 2^{2p-1}\left(\E\left[\left(k_{N}^{s}(\tau)\right)^{2p}\right] + \E\left[\left(q_{N}(\tau)\right)^{2p}\right]\right).
	\end{equation*}
	Since $q_{N}(\tau)$ is the modulus of a centered Poisson distribution with parameter $N\tau$, we may use estimates for the central moments of Poisson random variables to obtain that for each $r\geq1$,
	\begin{equation*}
		\E\left[ \left( q_{N}(\tau)\right)^{2r}\right] \;\lesssim\; (N\tau)^{r}.
	\end{equation*}
	We now use  the reverse triangle inequality we may bound the moments of $k_{N}^{s}(\tau)$ as  follows:
	\begin{equation*}
		\begin{split}
			\E\left[ \left(k_{N}^{s}(\tau)\right)^{2r}\right] & \;=\;  \E \left[ \left| q_{N}^{s}(\tau) - q_{N}(\tau)\right|^{2r}\right]\\ 
			& \;\leq\; \E\left[\left(F(Ns)-F(N(s-t))-F(N\tau) - \lfloor \tau N \rfloor + N\tau \right)^{2r}\right] \\ &\;\lesssim\; \E\left[\left( F(Ns)-F(N(s-\tau)) - F(N\tau)\right)^{2r}\right] + 1\,.
		\end{split}
	\end{equation*}
	Let $G_N(s,\tau) := F(Ns) - F(N(s-\tau)) - F(N\tau)$ for $s>\tau$. Since $F(t)$ is a Poisson process, the distribution of $G_N(s,\tau)$ depends on how $\tau$ and $s-\tau$ are ordered. If $s \leq 2 \tau$, we obtain that $F(Ns) - F(N\tau)$ and $F(N(s-\tau))$ are independent and $F(Ns) - F(N\tau) \sim F(N(s-\tau))$. So, in this case $G_N(s,\tau) \sim \mathrm{Skellam}(N(s-\tau),N(s-\tau))$ and we obtain\footnote{$\mathrm{Skellam}(\lambda_1,\lambda_2)$ is the distribution obtained from the difference $X_1-X_2$, where  $X_1\sim \mathrm{Poisson}(\lambda_1)$ and $X_2\sim \mathrm{Poisson}(\lambda_2)$ are independent random variables.}
\begin{equation*}
	\E\left[(G_N(s,\tau))^{2r}\right] \;\lesssim\; \left(N(s-\tau)\right)^{r}+ N(s-\tau)\,.
\end{equation*}
	Otherwise, if $s>2\tau$, we get that $F(Ns) - F(N(s-\tau))$ and $F(N\tau)$ are independent and $F(Ns) - F(N(s-\tau)) \sim F(N\tau)$, which lead us that $G_N(s,\tau) \sim \mathrm{Skellam}(N\tau,N\tau)$ and
	\begin{equation*}
		\E\left[(G_N(s,\tau))^{2r}\right] \;\lesssim\; \left(N\tau\right)^{r} + N\tau\,.
	\end{equation*}
	Now, joining the bounds we obtain the following moment estimate
	\begin{equation*}
		\E\big[ |D_{N}(s)|^{p}\big] \;\lesssim\; N^{-p/2} + e^{-Nc(s,\tau)}\,.
	\end{equation*}
		Applying the  Cauchy-Schwarz inequality and 
		Lemma~\ref{lem:momentsbound1},
	\begin{equation*}
		\begin{split}
			\E\bigg[\sup_{t \in [\tau,T]} |I_N(t)| \bigg] &\;\lesssim\; \E\bigg[ \int_{\tau}^T Y^N(s) |D^N(s)|ds \bigg] \\ 
			& \;\leq\; \bigg(\E\bigg[ \Big(\sup_{t \in [0,T]}Y^{N}(t)\Big)^2\bigg] \bigg)^{1/2} \bigg(\E\bigg[\Big(\int_{\tau}^{T} |D_{N}(s)|ds\Big)^2\bigg]\bigg)^{1/2}\\
			& \;\lesssim\; \bigg(\int_{\tau}^{T}\E\Big[ |D_{N}(s)|^2\Big]ds\bigg)^{1/2} \;\lesssim\; N^{-1/2}+ e^{-Nc(T,\tau)/2}\,.
		\end{split}
	\end{equation*}
	Then, applying the  Markov inequality leads us to
	\begin{equation*}
		\lim_{N\to\infty} \sup_{t\in [\tau,T]} |I_N(t)| \;=\; 0
	\end{equation*}
	in probability.
\end{proof}

\begin{proposition}\label{lem:Ztildelem}
We have for every $T>0$ that
\begin{equation*}
	\lim_{N\to\infty} \sup_{t\in [0,T]} \Big|\int_{0}^{t}Y^{N}(s)\widetilde{Z}^{N}(s)\,ds\Big|\;=\;0
\end{equation*}	
in probability.
	\end{proposition}
\begin{proof}
By Cauchy-Schwarz inequality and Lemma~\ref{lem:momentsbound3},
\begin{equation*}
	\E\Big[\sup_{t\in[0,T]}\Big|\int_{0}^{t}Y^{N}(s)\widetilde{Z}^{N}(s)\,ds\Big|\Big]
	\;\leq\; \int_0^T \E\big[Y^N(s)^2\big]^{1/2} \E\big[\widetilde{Z}^{N}(s)^2\big]^{1/2} ds
	\;\lesssim\; \frac{1}{N}\,,
\end{equation*}
as wanted.
\end{proof}

\subsection{Convergence at negative times and tightness}
In this section we deal with tightness of the sequence of processes $\{Y^N(t) : t \in [-\tau,T]\}_{N\geq 1}$. We start 
by studying the convergence of the process at negative times.

\begin{proposition}\label{lem:initialcondition}
	The entire initial condition converges uniformly on the interval $[-\tau,0]$, i.e, 
	\begin{equation*}
		\lim_{N \rightarrow \infty}\sup_{t \in [-\tau,0]}|Y^{N}(t)-\mu| \;=\; 0
	\end{equation*}
	in probability. In particular, the sequence $\{Y^N(t) : t \in [-\tau,0]\}_{N\geq 1}$ is tight in $D([-\tau,0], \mathbb{R})$ with respect to the Skorokhod topology.
\end{proposition}
\begin{proof}
	For each $t \in [-\tau,0]$, by applying Lemma~\ref{lem:timesdifferences} it holds almost surely that
	\begin{align*}
		|Y^{N}(t) - \mu| &\;=\; N^{-1}\big|\xi_{0}^{N}(F(N(\tau+t))-\tau N)-N\mu\big|\mathds{1}_{\left\{F(N(\tau+t))>\tau N\right\}}\\ 
		&\;\leq\;  N^{-1}H^{N}(T)\big|F(N(t+\tau))-\tau N\big|\mathds{1}_{\left\{F(N(\tau+t))>\tau N\right\}}\,.
	\end{align*}
	Now, taking the supremum in the interval $[-\tau,0]$ and using the fact that $F$ is non-decreasing, we obtain that almost surely
	\begin{align*}
		\sup_{t \in [-\tau,0]}|Y^{N}(t)-\mu| &\;\leq\; N^{-1}H^{N}(T)\sup_{s \in [0,\tau]}|F(Ns)-\tau N|\mathds{1}_{\left\{F(Ns) >\tau N)\right\}} \\ 
		&\;=\; N^{-1}H^{N}(T)\sup_{s \in [0,\tau]}(F(Ns)-\tau N)_{+} \\ 
		&\;=\; N^{-1}H^{N}(T)(F(N\tau)-\tau N)_{+}
	\end{align*}
	Taking expectations, using applying Cauchy-Schwarz and Lemmas  \ref{lem:momentsbound1} and \ref{lem:momentsbound2}, we obtain
	\begin{align*}
		\E\left[|Y^{N}(t) - \mu|\right] &\;\lesssim\; N^{-1}(\E\left[(F(N\tau)-\tau N)^2\right])^{1/2}\;\lesssim\; N^{-1/2}\,,
	\end{align*}
 concluding the proof.
\end{proof}

\begin{proposition}\label{lem:jumps}
	For each $T>0$,
	\begin{equation*}
		\lim_{N \rightarrow \infty}\sup_{t \in [-\tau,T]}|Y^{N}(t)-Y^{N}(t^{-})| \;=\; 0
	\end{equation*}
	in probability.
\end{proposition}
\begin{proof}
	For each $N\in \N$, we may use Lemmas~\ref{lem:timesdifferences}, \ref{lem:momentsbound1} and \ref{lem:momentsbound2} to obtain the bound
	\begin{align*}
		\E\bigg[\sup_{t \in [-\tau,T]}|Y^{N}(t)\!-\!Y^{N}(t^{-})|\bigg] \leq N^{-1} \E\bigg[ H^{N}(T)\!\sup_{t \in [0,T]}\!\!(F(Nt)\!-\!F((Nt)^{-})\bigg] \lesssim N^{-1}.
	\end{align*}
	Hence, using Markov's inequality we may conclude.
\end{proof}

\begin{proposition}\label{lem:tight}
	The sequence $\{Y^N(t) : t \in [-\tau,T]\}_{N\geq 1}$ is tight in $D([-\tau,T], \mathbb{R})$ with respect to the Skorokhod topology.
\end{proposition}

\begin{proof}
	First of all we note that it is enough to prove tightness in $D([0,T], \bb R)$ since for negative times the result follows from Lemma~\ref{lem:initialcondition}.
	A sufficient criterion for the tightness of a sequence of processes $W^{N} = \{W^N(t) : t \in [0,T]\}$ in $D([0,T], \mathbb{R})$ is that the following three conditions hold:
	\begin{itemize}
		\item For every $\varepsilon > 0$, there exists $M > 0$ such that
		\begin{equation*}
			\sup_{N \in \mathbb{N}} \mathbb{P}\left( \sup_{t \in [0,T]} |W^{N}(t)| > M \right) \;\leq\; \varepsilon\,.
		\end{equation*}
		
		\item There exist $p \geq 0$, $\alpha > 1/2$, and a non-decreasing continuous function $g : [0,T] \to \mathbb{R}$ such that, for all $r < s < t$ and all $N \in \mathbb{N}$,
		\begin{equation*}
			\mathbb{E}\Big[ |W^{N}(s) - W^{N}(r)|^{2p} \, |W^{N}(t) - W^{N}(s)|^{2p} \Big]
			\;\leq\; \left( g(t) - g(r) \right)^{2\alpha}.
		\end{equation*}
	\item For all $\eps , \eta>0$, there exists $\delta>0$ such that
	\begin{equation*}
		\begin{aligned}
			&\bb P\big[|Y^N(\delta)-Y^N(0)|\geq \eps\big]\;\leq\; \eta\quad\text{ and } \quad \bb P\big[|Y^N(1-)-Y^N(1-\delta)|\;\geq\; \eps \big]\leq \eta\,.
		\end{aligned}
	\end{equation*}
	\end{itemize}
	Even though it might be well known, we did not find in the literature a clear place where the sufficiency of these three conditions are stated. However, as we will argue below this follows from results in~\cite[Chapter 13]{Bill}. For now we take this as granted.
	The first condition follows directly from Lemma~\ref{lem:momentsbound1}. 
	For the second condition, by Lemma~\ref{lem:timesdifferences}, we obtain that for all $s,t \in [0,T]$, almost surely,
	\begin{equation*}
		|Y^{N}(t) - Y^{N}(s)| \;\leq\; N^{-1} H^{N}(T) \, F([Ns,Nt])\,,
	\end{equation*}
	where $F([Ns,Nt])$ denotes the number of jumps of the process $X^N$ in the time interval $[Ns,Nt]$. Hence, for $r < s < t$ and $p \geq 0$,
	\begin{equation}\label{eq:Ydif}
		\begin{aligned}
			\mathbb{E}\big[ &|Y^{N}(s) - Y^{N}(r)|^{2p} \, |Y^{N}(t) - Y^{N}(s)|^{2p} \big] \\
			&\leq N^{-4p} \, \mathbb{E}\Big[ (H^{N}(T))^{4p} \, (F([Nr,Ns]))^{2p} \, (F([Ns,Nt]))^{2p} \Big].
		\end{aligned}
	\end{equation}
	Applying Hölder's inequality together with Lemmas~\ref{lem:momentsbound1} and~\ref{lem:momentsbound2}, we obtain that, for every $q \geq 1$,
	\begin{equation*}
		\begin{aligned}
			&\mathbb{E}\Big[ |Y^{N}(s) - Y^{N}(r)|^{2p} \, |Y^{N}(t) - Y^{N}(s)|^{2p} \Big] \\
			&\lesssim N^{-4p} \Big( \mathbb{E}\big[(F([Nr,Ns]))^{2pq} (F([Ns,Nt]))^{2pq}\big] \Big)^{1/q}.
		\end{aligned}
	\end{equation*}
	Since $F$ is a Poisson process with rate one, the random variables $F([Nr,Ns])$ and $F([Ns,Nt])$ are independent. Therefore,
	\begin{equation*}
		\begin{aligned}
			\mathbb{E}\big[ &|Y^{N}(s) - Y^{N}(r)|^{2p} \, |Y^{N}(t) - Y^{N}(s)|^{2p} \big] \\
			&\lesssim \Big(N|s-r| + N^{2pq}|s-r|^{2pq}\Big)^{1/q}
			\Big(N|t-s| + N^{2pq}|t-s|^{2pq}\Big)^{1/q}.
		\end{aligned}
	\end{equation*}
	To bound this expression from above, we replace $|s-r|$ and $|t-s|$ by $|t-r|$, obtaining
	\begin{equation*}
		N^{-4p} \left( N^{4p}|t-r|^{4p} + N^{2p+1/q}|t-r|^{2p+1/q} + N^{2/q}|t-r|^{2/q} \right).
	\end{equation*}
	Thus, if $p \geq 1/2$ and $1 < q < 2$, we conclude that
	\begin{equation*}
		\mathbb{E}\left[ |Y^{N}(s) - Y^{N}(r)|^{2p} \, |Y^{N}(t) - Y^{N}(s)|^{2p} \right]
		\;\lesssim\; |t-r|^{4p} + |t-r|^{2p+1/q} + |t-r|^{2/q},
	\end{equation*}
	which verifies the second condition.
	The third condition follows from~\eqref{eq:Ydif} and from an application of Hölder's inequality together with Markov's inequality.
	
	We now show that these three conditions imply tightness. By \cite[Theorems 13.2 and 13.3]{Bill}, assuming our first and third conditions, it suffices to prove that for every $\varepsilon, \eta > 0$, there exists $\delta > 0$ such that
	\begin{align*}
		\mathbb{P}\left( \sup_{\substack{r \leq s \leq t \\ t-r \leq \delta}}
		\left\{ |W^{N}(s)-W^{N}(r)| \wedge |W^{N}(t)-W^{N}(s)| \right\}
		\geq \varepsilon \right) \;\leq\; \eta\,.
	\end{align*}
	Let $v(\delta) = \lfloor 1/\delta \rfloor$, set $t_{v(\delta)} = T$, and define $t_j = j\delta T$ for $0 \leq j < v(\delta)$. If $t-r \leq \delta$, then $r$ and $t$ either lie in the same interval $[t_{j-1}, t_j]$ or in two adjacent intervals. In both cases, they belong to some interval $[t_{j-1}, t_{j+1}]$ with $1 \leq j \leq v(\delta)-1$. Therefore,
	\begin{align*}
		\mathbb{P}\left( \sup_{\substack{r \leq s \leq t \\ t-r \leq \delta}}
		\left\{ |W^{N}(s)-W^{N}(r)| \wedge |W^{N}(t)-W^{N}(s)| \right\}
		\geq \varepsilon \right)
		\;\leq\; \sum_{j=1}^{v(\delta)-1} \mathbb{P}(\ell_j \geq \varepsilon)\,,
	\end{align*}
	where
	\begin{equation*}
		\ell_j = \sup_{\substack{r \leq s \leq t \\ r,s,t \in [t_{j-1}, t_{j+1}]}}
		 \big|W^{N}(s)-W^{N}(r)\big| \wedge \big|W^{N}(t)-W^{N}(s)\big| \,.
	\end{equation*}
	By Markov's inequality and the second condition, for all $r < s < t$ and $\varepsilon > 0$,
	\begin{align*}
		&\mathbb{P}\left( |W^{N}(s)-W^{N}(r)| \wedge |W^{N}(t)-W^{N}(s)| \geq \varepsilon \right) \\
		&\leq\; \varepsilon^{-4p}
		\mathbb{E}\Big[ \left( |W^{N}(s)-W^{N}(r)| \wedge |W^{N}(t)-W^{N}(s)| \right)^{4p} \Big] \\
		&\leq\; \varepsilon^{-4p}
		\mathbb{E}\Big[ |W^{N}(s)-W^{N}(r)|^{2p} \, |W^{N}(t)-W^{N}(s)|^{2p} \Big] \\
		&\leq\; \varepsilon^{-4p} (g(t)-g(r))^{2\alpha}\,.
	\end{align*}
	By \cite[Theorem 10.3]{Bill}, there exists a constant $K_{p,\alpha} > 0$ such that, for all $1 \leq j \leq v(\delta)-1$,
	\begin{equation*}
		\mathbb{P}(\ell_j \geq \varepsilon)
		\;\leq\; \varepsilon^{-4p} K_{p,\alpha} (g(t_{j+1}) - g(t_{j-1}))^{2\alpha}\,.
	\end{equation*}
	Summing over $j$, we obtain
	\begin{align*}
		&\mathbb{P}\left( \sup_{\substack{r \leq s \leq t \\ t-r \leq \delta}}
		\left\{ |W^{N}(s)-W^{N}(r)| \wedge |W^{N}(t)-W^{N}(s)| \right\}
		\geq \varepsilon \right) \\
		&\leq\; \varepsilon^{-4p} K_{p,\alpha}
		\sum_{j=1}^{v(\delta)-1} (g(t_{j+1}) - g(t_{j-1}))^{2\alpha} \\
		&\leq\; \varepsilon^{-4p} K_{p,\alpha}
		\sum_{j=1}^{v(\delta)-1} (g(t_{j+1}) - g(t_{j-1}))
		\max_{1 \leq j \leq v(\delta)-1}
		(g(t_{j+1}) - g(t_{j-1}))^{2\alpha-1} \\
		&\leq\; 2 K_{p,\alpha} (g(T) - g(0)) \varepsilon^{-4p}
		\left( \omega_g(2\delta) \right)^{2\alpha - 1}\,,
	\end{align*}
	where $\omega_g$ denotes the modulus of continuity of $g$. Since $g$ is uniformly continuous and $\alpha > 1/2$, we can choose $\delta$ sufficiently small so that
	\begin{equation*}
		2 K_{p,\alpha} (g(T) - g(0)) \varepsilon^{-4p}
		\left( \omega_g(2\delta) \right)^{2\alpha - 1}
		\;\leq\; \eta\,.
	\end{equation*}
	This completes the proof.
\end{proof}

\section{Computational Simulations}\label{s4}
We present next some Python simulations of the delayed random walk  defined in Subsection~\ref{s2}. It is known that the Delayed Logistic ifferential equation of the form \eqref{eq:Delayed}
 exhibits a Hopf bifurcation when $\tau R = \pi/2$,  see \cite{Hopf_bifurcation}. In our model $R=1/2$, hence for delays $\tau< \pi$ any solution converges to the constant solution $u\equiv 1$, while for $\tau> \pi$, a periodic solution becomes attractive. This is observed in our computational simulations, (the code is available at \cite{Delayed_RW}). Note that  these simulations are done in a very large state-space, namely $\Omega_N = \bb R_{\geq 0}^{\lfloor \tau N\rfloor +1}$, where $N=50,\!000$ or $N=100,\!000$. 
The curves obtained fit reasonably well the graphs of the respective  solutions of the Hutchinson delayed differential equation according to the chosen parameters, see \cite[Chapters 5 and 6]{Hal} for instance. 
\newcommand{\constante}{0.6}
\begin{figure}[H]
	\centering
	\includegraphics[width=\constante\linewidth]{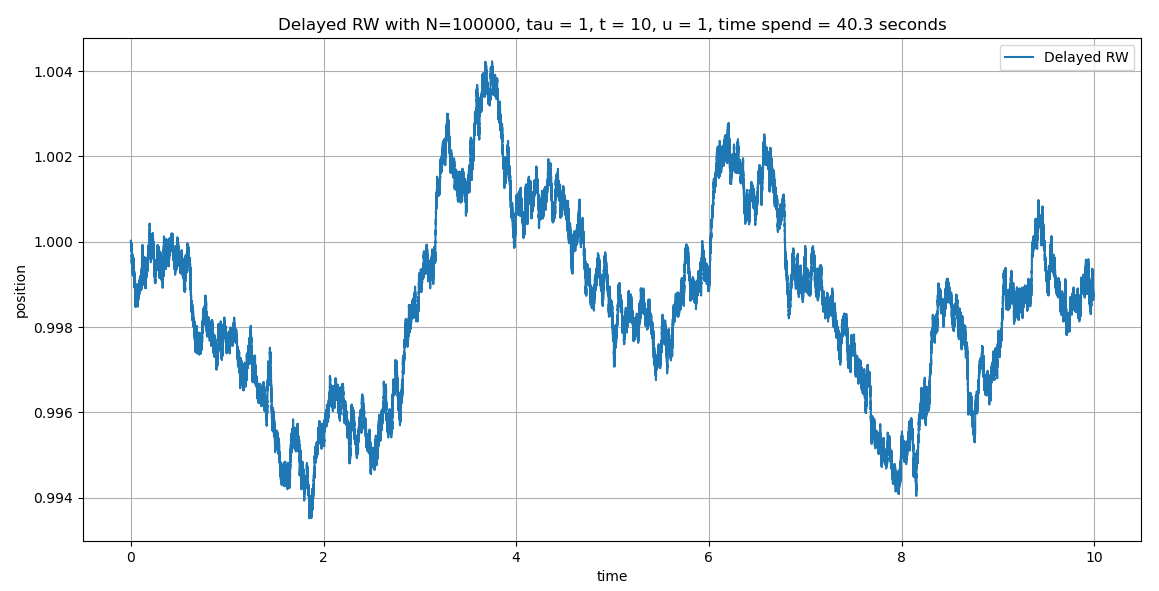}
	\caption{This simulation starts from the invariant profile $u=1$. 	Therefore, one expects that output of the simulation of the Markov chain to be approximately constant. This precisely what we see in the picture noting the drawing's scale: the oscillation in the $y$-axis is less than $0.05$ around the invariant value $1$, hence less than $0.5\%$ with respect to the invariant value $u=1$.}
	\label{fig:figure3}
\end{figure}

\begin{figure}[H]
	\centering
	\includegraphics[width=\constante\linewidth]{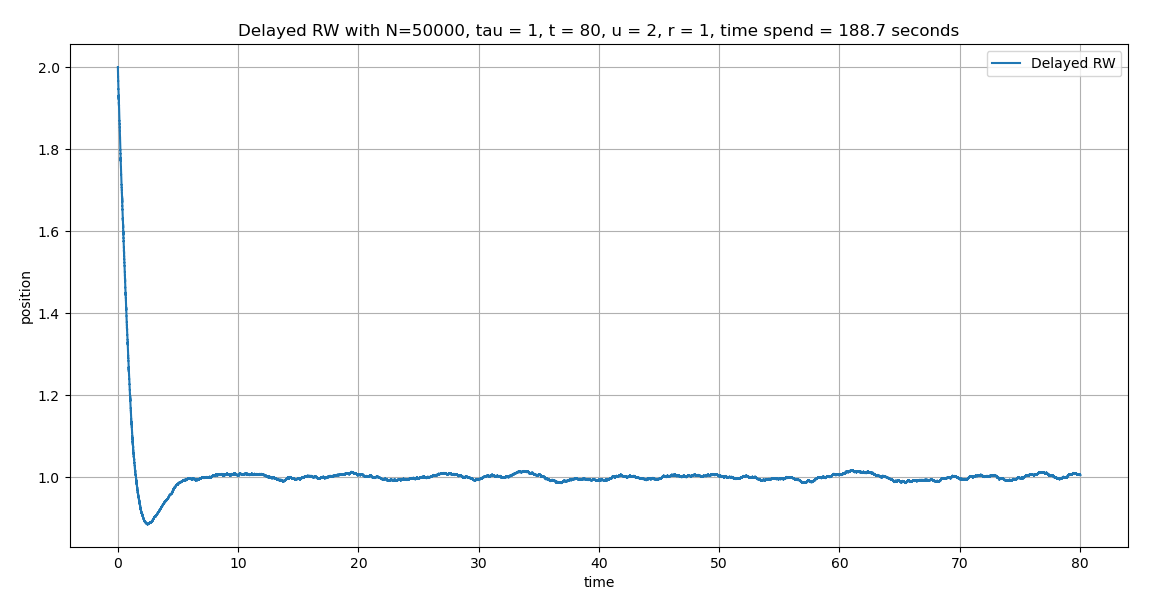}
	\caption{This simulation has delay $\tau = 1$ and starts at $u=2$. Since $\tau < \pi$, which is the critical value of the Hopf's bifurcation of the delayed Logistic Equation, the invariant profile $u\equiv 1$ is attractive. Moreover, contrarily to the classical logistic differential equation, which does not go below one once starting from $u>1$, the delayed logistic does, as we can see in the picture.}
	\label{fig:figure1}
\end{figure}

\begin{figure}[H]
	\centering
	\includegraphics[width=\constante\linewidth]{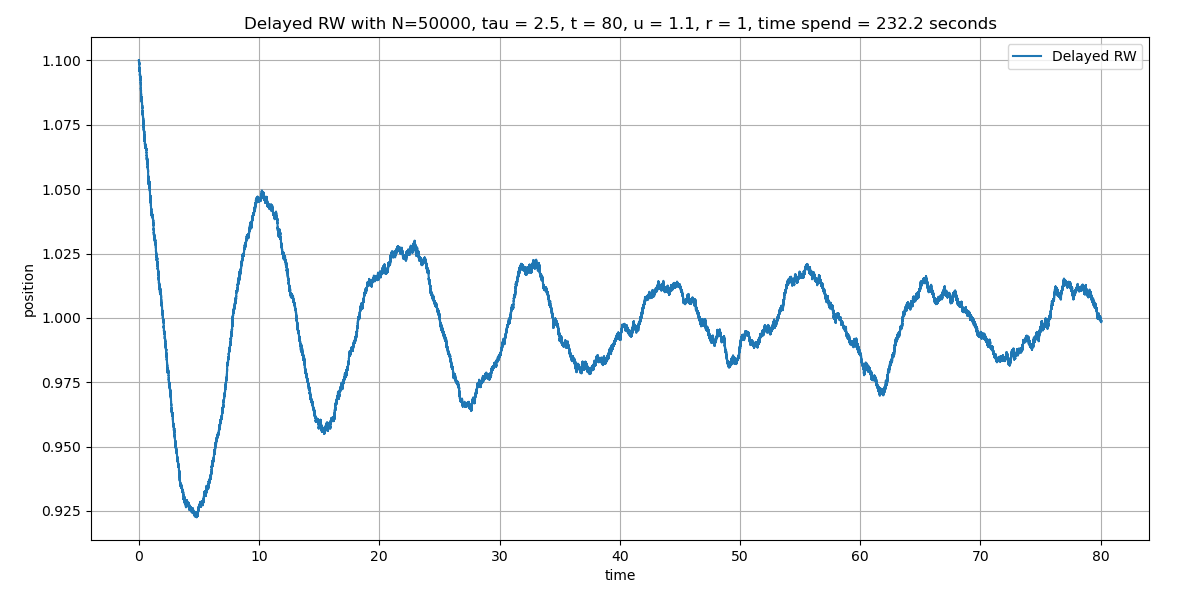}
	\caption{This simulation has delay $\tau = 2.5<\pi$, so it lies in the subcritical regime and starts at $u=1.1$. The proximity of the starting point with the invariant profile (constant equal to one) creates some noise in the simulation.}
	\label{fig:figure10a}
\end{figure}

\begin{figure}[H]
	\centering
	\includegraphics[width=\constante\linewidth]{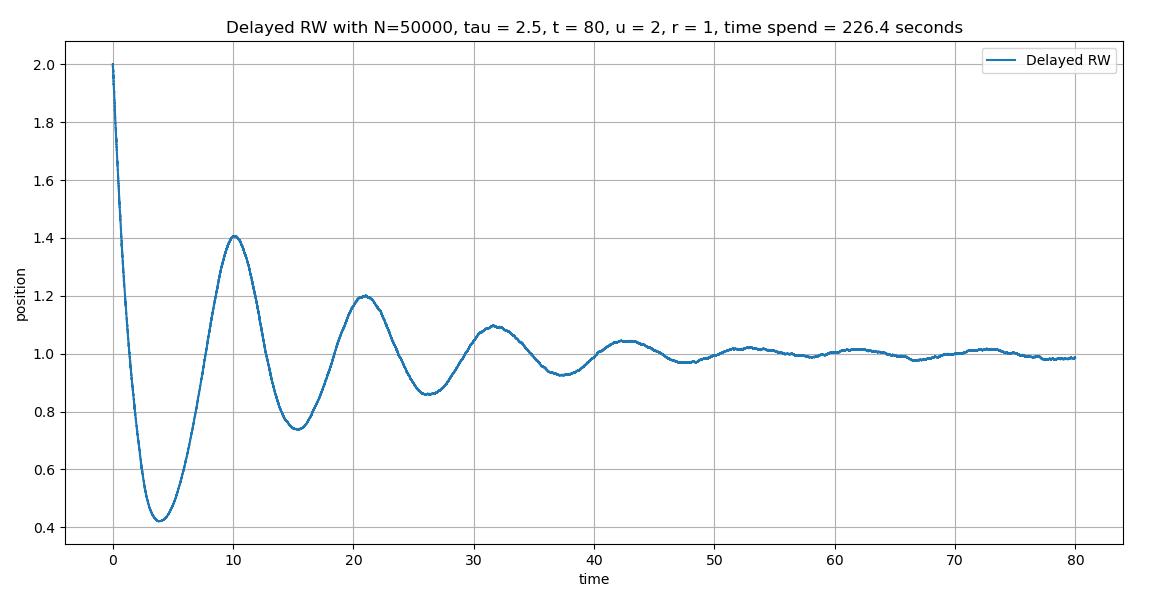}
	\caption{Another picture in the subcritical regime with delay $\tau = 2.5< \pi$, starting at $u=2$. The simulation is smoother in comparison with the previous picture. We believe is due to a bigger distance of the starting point to the invariant profile equal to one.}
	\label{fig:figure9a}
\end{figure}

\begin{figure}[H]
	\centering
	\includegraphics[width=\constante\linewidth]{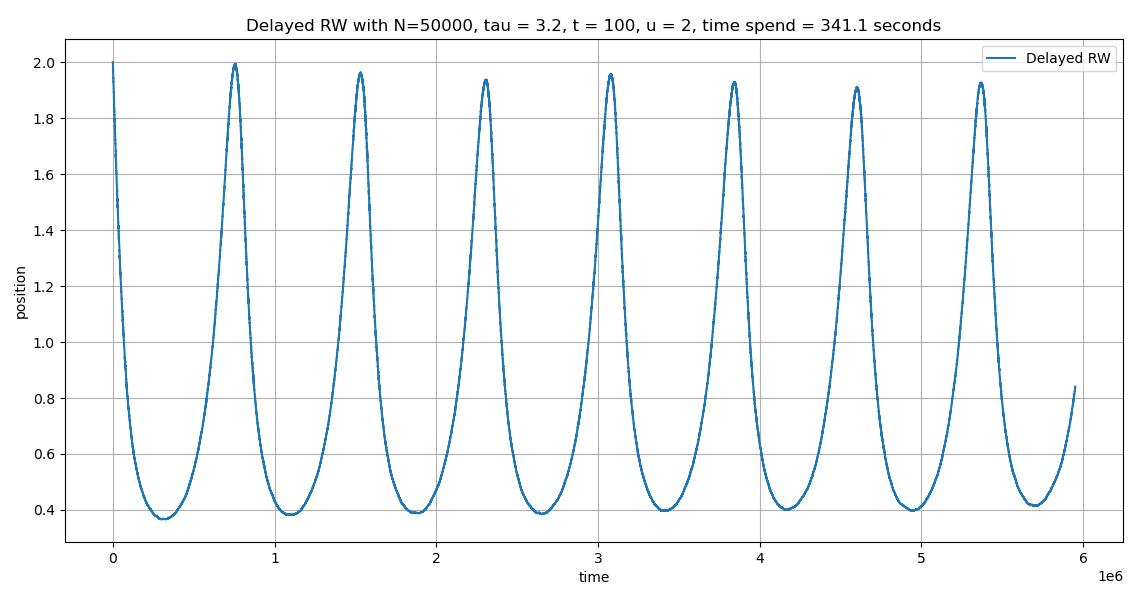}
	\caption{Here $\tau = 3.2> \pi$, starting at $u=2$. In this supercritical regime, a periodic orbit becomes attractive.}
	\label{fig:figure15}
\end{figure}

\section*{Acknowledgements}
D.E.\ was supported by the National Council for Scientific and Technological Development - CNPq via a Bolsa de Produtividade  303348/2022-4. D.E.\ moreover acknowledge support by the Serrapilheira Institute (Grant Number Serra-R-2011-37582). D.E, T.F.\ and M.J.\  acknowledge support by the National Council for Scientific and Technological Development - CNPq via a Universal Grant (Grant Number 406001/2021-9). D.E. and
T.F.\ acknowledge support by the CNPq via a Universal Grant (Grant Number 401314/2025-1). T.F.\ was supported by the National Council for Scientific and Technological Development - CNPq via a Bolsa de Produtividade number 306554/2024-0. E.B.\ thanks CAPES.

\bibliography{bibliografia}
\bibliographystyle{plain}
\end{document}